\documentclass[11pt,letterpaper,reqno]{amsart}
\usepackage{amssymb}
\usepackage{amsmath}
\usepackage{amsthm}
\usepackage{amsfonts}
\usepackage{bbm}
\usepackage{enumitem} 
\usepackage{booktabs}
\usepackage{graphicx}
\usepackage{xcolor}
\usepackage[T1]{fontenc}
\usepackage{doi}
\usepackage{comment} 
\usepackage{hyperref}
\hypersetup{pdfstartview={FitH}}
\usepackage{bookmark}
\usepackage[capitalize,noabbrev]{cleveref}

\newtheorem{theorem}{Theorem}[section]

\theoremstyle{definition}

\begin{document}

\title{Consecutive integers free of certain prime factors}

\author{Wouter van Doorn}
\address{Groningen, the Netherlands}
\email{wonterman1@hotmail.com}

\author{Quanyu Tang}
\address{School of Mathematics and Statistics, Xi'an Jiaotong University, Xi'an 710049, P. R. China}
\curraddr{School of Mathematical Sciences, University of Science and Technology of China, Hefei 230026, P.~R.~China}
\email{tangquanyu827@gmail.com}

\subjclass[2020]{Primary 11N25; Secondary 11N05, 11A41.}

\keywords{Erd\H{o}s problem, consecutive integers, prime factors}

\begin{abstract}
Let $n_k$ denote the least integer $n>2k$ such that $(n-k)(n-k+1)\cdots(n-1)$ is not divisible by any prime in the interval $(k,2k)$. Confirming a conjecture of Erd\H{o}s, we prove that, for all sufficiently large $k$,
$$
n_k > e^{\frac{\log^2 k}{20 \log \log k}}.
$$
\end{abstract}

\maketitle

\section{Introduction}
In~\cite{Erdos1979} Erd\H{o}s writes\footnote{Notation slightly altered to match notation in~\cite{ErdosGraham1980}.}
\begin{quote}
\emph{Several times during my long life, I was led to questions of the following type. Estimate, as well as you can, the size of the smallest integer $n_k > 2k$ for which $\prod_{1 \le i \le k} (n_k - i)$ has no prime factor $p$ satisfying $k < p < 2k$. I would expect that $n_k > k^d$ for every $d$ if $k > k_0(d)$, but that $n_k < e^{\varepsilon k}$ for every $\varepsilon > 0$ if $k > k_0(\varepsilon)$. However, I could prove nothing non-trivial.}
\end{quote}

Erd\H{o}s and Graham repeated this request in~\cite{ErdosGraham1980}, and added
\begin{quote}
\emph{We can prove $n_k > k^{1+c}$ but no doubt much more is true.}
\end{quote}

Estimating $n_k$ is now recorded as Erd\H{o}s Problem \#451 on Bloom's website~\cite{ErdosProblems451}, and no proofs of non-trivial bounds have appeared in the literature, as far as the authors are aware. 

For a quick upper bound, we have $$n_k \le \prod_{\substack{k < p < 2k \\ p \text{ prime}}} p = e^{(1+o(1))k}$$ by the prime number theorem. However, a natural heuristic suggests that the true size of \(n_k\) is smaller, and plausibly satisfies
\[
    \exp\left(\frac{c_1k}{\log k}\right) < n_k < \exp\left(\frac{c_2k}{\log k}\right),
\]
for some absolute positive constants $c_1$ and $c_2$ and all $k > 1$. The details of this heuristic, together with related numerical evidence, are discussed in~\cite{TangHeuristics}; computed values of \(n_k\) are also recorded in OEIS sequence A386620~\cite{OEISA386620}. 

In this paper we focus on lower bounds, and prove Erd\H{o}s' conjecture that $n_k$ grows superpolynomially fast, by showing that $n_k > e^{\frac{c\log^2 k}{\log \log k}}$ holds for some absolute constant $c$ and all $k > 1$. More precisely, with $\theta \in \left(\frac{2}{5}, \frac{3}{5}\right)$ a constant such that, for all sufficiently large $k$, the interval $I = I_k := (k, k + k^{\theta})$ contains $\gg_{\theta} \frac{k^{\theta}}{\log k}$ primes (which exists by e.g.~\cite{BakerHarmanPintz2001}), this paper is then dedicated to proving the following theorem.

\begin{theorem} \label{main}
For all sufficiently large integers $k$ and all $n \in \mathbb{N}$ with $2k < n \le e^{\frac{\log^2 k}{20 \log \log k}}$, the product $(n-k)\cdots(n-1)$ is divisible by some prime $p \in (k, k + 3k^{\theta})$.
\end{theorem}

For the proof of this lower bound we turn to the similar problem of estimating the smallest $n > k+1$ such that the binomial coefficient $\binom{n}{k}$ has no prime factors smaller than or equal to $k$. This function was studied by Ecklund, Erd\H{o}s and Selfridge in~\cite{EES1974}, where they proved that there exists an absolute constant $c > 0$ such that the bounds $k^{1+c} < n < e^{(1+o(1))k}$ hold for all $k \in \mathbb{N}$. Estimating this quantity is now Erd\H{o}s Problem \#1095~\cite{ErdosProblems1095}, and the best known lower bound of $e^{c \log^2 k}$ (for some absolute constant $c > 0$) was proven by Konyagin~\cite{Konyagin1999}. It is Konyagin's proof that we will imitate in the coming sections.

\section{Small \texorpdfstring{$n$}{n}}
First assume $2k < n \le \frac{1}{2}k^{2-\theta}$, which implies that there exists an integer $m$ with $2 \le m \le \frac{1}{2}k^{1-\theta}$ such that $mk < n \le (m+1)k$. Then there are two cases, either $$mk < n < mk + mk^{\theta} \qquad \text{or} \qquad mk + mk^{\theta} \le n \le (m+1)k.$$ In the first case, let $p$ be a prime with $$p \in \left(k + \frac{m}{m-1}k^{\theta}, k + \frac{2m-1}{m-1}k^{\theta} \right) \subset (k, k + 3k^{\theta}).$$ We then claim that $(m-1)p$ occurs as a factor in $(n-k)\cdots(n-1)$. Indeed, on the one hand we have $$(m-1)p > (m-1)k + mk^{\theta} > n-k,$$ while on the other hand we have $$(m-1)p < (m-1)k + (2m-1)k^{\theta} < (m-1)k + k < n.$$ In the second case, let $p$ be a prime with $p \in I$. We then have $$mp > mk \ge n-k$$ and $$mp < mk + mk^{\theta} \le n,$$ so that in this case $mp$ occurs as a factor in $(n-k)\cdots(n-1)$.

\section{Medium \texorpdfstring{$n$}{n}}
The medium case is where $\frac{1}{2}k^{2-\theta} < n \le \frac{k^2}{\log^2 k}$. With $\|x\|$ the distance from a real number $x$ to the nearest integer, let $K$ be the number of integers $m \in I$ such that $\| \frac{n}{m} \| < \frac{1}{k^{1-\theta}}$. We then claim that it suffices to show $K = o\left(\frac{k^{\theta}}{\log k}\right)$. To see why this is indeed sufficient, we note that it implies that there exists a prime $p \in I$ with $\| \frac{n}{p} \| \ge \frac{1}{k^{1 - \theta}}$, by definition of $\theta$. And this in turn implies that reducing $n$ modulo $p$ gives $$n \bmod p \in [k^{\theta}, p - k^{\theta}] \subset [1, k],$$  so that $p$ divides $(n-k)\cdots(n-1)$. 

For all $m \in I$, we have $$\frac{n}{m} \in \left(\frac{n}{k + k^{\theta}}, \frac{n}{k} \right) \subseteq J := \left(\left \lfloor \frac{n}{k + k^{\theta}} \right \rfloor, \left \lceil \frac{n}{k} \right \rceil \right)$$ with $$\left|J \cap \mathbb{Z}\right| \le 3 + \frac{n}{k} - \frac{n}{k + k^{\theta}} = 3 + \frac{nk^{\theta}}{k(k + k^{\theta})} < 3 + \frac{n}{k^{2-\theta}}.$$ For every integer $h \in J$, if $\frac{n}{m} \in \left(h - \frac{1}{k^{1 - \theta}}, h + \frac{1}{k^{1 - \theta}} \right)$, then $$\left|m - \frac{n}{h} \right| = \frac{m}{h} \left|\frac{n}{m} - h \right| < \frac{2k}{\frac{n}{k}} \frac{1}{k^{1-\theta}} = \frac{2k^{1+\theta}}{n}.$$ This implies that for every integer $h \in J$ there are at most $1 + \frac{4k^{1+\theta}}{n}$ values of $m$ such that $\left| \frac{n}{m} - h \right| < \frac{1}{k^{1-\theta}}$, which shows 
\begin{align*}
K &< \left(1 + \frac{4k^{1+\theta}}{n}\right)|J \cap \mathbb{Z}| \\
&< 3 + \frac{12k^{1+\theta}}{n} + \frac{n}{k^{2 - \theta}} + 4k^{2\theta - 1}  \\
&< 3 + 28k^{2\theta-1} + \frac{\frac{k^2}{\log^2 k}}{k^{2 - \theta}} \\
&= o\left(\frac{k^{\theta}}{\log k}\right).
\end{align*}

\section{Intermezzo on a theorem by Konyagin}
To show that $K = o\left(\frac{k^{\theta}}{\log k}\right)$ also holds for all $n$ with $\frac{k^2}{\log^2 k} < n \le e^{\frac{\log^2 k}{20 \log \log k}}$ we will make use of the following general upper bound on $K$.

\begin{theorem} \label{konyaginappl}
Let $r$ be an integer with $2 \le r \le \frac{1}{2}k^{1 - \theta}$ and let $\lambda \ge 1$ be arbitrary. We then have 
\begin{equation} \label{maxPoints}
K \ll k^{\theta} \left(\left(\frac{nr!\lambda^r}{k^{r+1}} \right)^{\frac{1}{2r-1}} + \left(\frac{k^{r+\theta}}{nr! \lambda^r} \right)^{\frac{1}{r-1}} + \left(\frac{(r+1)\lambda}{k} \right)^{\frac{1}{2r}} \right) + r\lambda.
\end{equation}
\end{theorem}

\begin{proof}
We apply~\cite[Theorem 2]{Konyagin1999}, with
\begin{align*}
N &:= k^{\theta}, \\
W &:= 1, \\
f(x) &:= \frac{(-1)^r n}{k+x}, \\
D_r &:= \frac{nr!}{k^{r+1}}, \\
\delta &:= \frac{1}{k^{1-\theta}}.
\end{align*}

The only non-trivial conditions one has to check are the required bounds in~\cite{Konyagin1999} on the iterated derivatives of $f$. But as $f^{(i)}(x) = \frac{(-1)^{r+i}ni!}{(k+x)^{i+1}}$, the upper bounds $f^{(r)}(x) \le D_r$ and $|f^{(r+1)}(x)| \le D_{r+1}$ quickly follow. To see why $f^{(r)}(x) \ge \frac{1}{2}D_r$ also holds, we use the assumption $r \le \frac{1}{2}k^{1 - \theta}$, as follows.
\begin{align*}
f^{(r)}(x) &= \frac{nr!}{(k+x)^{r+1}} \\
&\ge \frac{nr!}{\left(k\left(1 + \frac{1}{k^{1 - \theta}}\right)\right)^{r+1}} \\
&= \frac{D_r}{\left(1 + \frac{1}{k^{1 - \theta}}\right)^{r+1}} \\
&> \frac{D_r}{\left(1 + \frac{1}{2r}\right)^{r+1}} \\
&> \frac{1}{2}D_r. \qedhere
\end{align*}
\end{proof}

\section{Medium-large \texorpdfstring{$n$}{n}} \label{mediumlarge}
The medium-large case is where $\frac{k^2}{\log^2 k} < n \le \frac{1}{2}k^{2+\theta}$. We then apply Theorem \ref{konyaginappl} with $\lambda := \sqrt{\frac{k^{2 + \theta}}{2n}}\log k$ and $r := 2$, which gives
\begin{align*}
K &\ll k^{\theta} \left(\left(\frac{nr!\lambda^r}{k^{r+1}} \right)^{\frac{1}{2r-1}} + \left(\frac{k^{r+\theta}}{nr! \lambda^r} \right)^{\frac{1}{r-1}} + \left(\frac{(r+1)\lambda}{k} \right)^{\frac{1}{2r}} \right) + r\lambda \\
&= k^{\theta}  \left(k^{\frac{\theta - 1}{3}} (\log k)^\frac{2}{3} + \frac{1}{\log^2 k} + \left(\frac{(r+1)\lambda}{k} \right)^{\frac{1}{2r}} \right) + r\lambda \\
&\ll k^{\theta} \left(k^{\frac{\theta - 1}{3}} (\log k)^\frac{2}{3}  + \frac{1}{\log^2 k} + k^{\frac{\theta - 2}{8}} \sqrt{\log k} \right) + k^{\frac{\theta}{2}} \log^{2} k \\
&= o\left(\frac{k^{\theta}}{\log k} \right).
\end{align*}

\section{Large \texorpdfstring{$n$}{n}}
Finally, we deal with the case of large $n$. More precisely, let us take $c := \frac{1}{20}$ and assume $\frac{1}{2}k^{2 + \theta} < n \le e^{\frac{c \log^2 k}{\log \log k}}$. With $r$ defined as the smallest positive integer such that $nr! \le k^{r+\theta}$, we have $r \ge 3$ by the assumption on $n$. Furthermore, in general we see that such an $r$ does indeed exist and $$r \le r_0 := \left \lceil \frac{2c \log k}{\log \log k} \right \rceil \le \frac{1}{2}k^{1 - \theta},$$ as
\begin{align*}
k^{r_0+\theta} &> k^{\frac{2c \log k}{\log \log k}} \\
&\ge ne^{\frac{c \log^2 k}{\log \log k}} \\
&> ne^{r_0 \log r_0} \\
&> nr_0!.	
\end{align*}

We further define $\lambda := \left(\frac{k^{r+1-(1-\theta)\frac{2r-1}{3r-2}}}{nr!} \right)^{\frac{1}{r}},$ which is larger than $1$ by definition of $r$. Hence, Theorem \ref{konyaginappl} applies, and with this definition of $\lambda$, the first two terms within the brackets in Equation \eqref{maxPoints} are both equal to $k^{\frac{\theta-1}{3r-2}}.$ Making use of the inequalities $$r \le r_0, \qquad 3r_0-2 < \frac{7c \log k}{\log \log k} \qquad \text{and} \qquad \theta < \frac{3}{5},$$ we see that these first two terms are smaller than $$k^{\frac{(\theta-1)\log \log k}{7c \log k}} = (\log k)^{\frac{\theta-1}{7c}} = o\left(\frac{1}{\log k} \right),$$ by our choice of $c$. As for the third term, thanks to the minimality of $r$ we know that the inequality $n(r-1)! > k^{r-1+\theta}$ holds. This gives $$\lambda = \left(\frac{k^{r+1-(1-\theta)\frac{2r-1}{3r-2}}}{nr!} \right)^{\frac{1}{r}} < k^{\frac{2-\theta-(1-\theta)\frac{2r-1}{3r-2}}{r}} < k^{\frac{2-\theta}{r}}.$$ Hence, as $$r+1 \le r_0 + 1 < k^{\frac{\theta}{r_0}} \le k^{\frac{\theta}{r}}$$ if $k$ is large enough, the third term is bounded by 
\begin{align*}
    \left(\frac{(r+1)\lambda}{k} \right)^{\frac{1}{2r}} &< k^{\frac{1}{r^2} - \frac{1}{2r}} \\
    &\le k^{\frac{-1}{6r}} \\
    &\le k^{\frac{-1}{6r_0}} \\
    &< (\log k)^{\frac{-1}{13c}} \\
    &= o\left(\frac{1}{\log k} \right).
\end{align*} 
For the final additive term of $\lambda r$ we also apply the minimality of $r$, which implies $$\lambda r < k^{\frac{2-\theta-(1-\theta)\frac{2r-1}{3r-2}}{r}} r \le k^{\frac{9 - 2\theta}{21}} r_0 =  o\left(\frac{k^{\theta}}{\log k} \right),$$ by the assumption $\theta > \frac{2}{5} > \frac{9}{23}$ and the fact that $\frac{2-\theta-(1-\theta)\frac{2r-1}{3r-2}}{r}$ as a function over $r \in [3, r_0]$ is maximized at $r = 3$. As the full range of integers $n$ with $2k < n \le e^{\frac{\log^2 k}{20 \log \log k}}$ is now dealt with, this finishes the proof of Theorem \ref{main}.

\section*{Declaration of AI usage}
Even though this paper is completely human-written, the core idea of applying the same arguments used in~\cite{Konyagin1999} in order to prove the lower bound $n_k \ge e^{\frac{c \log^2 k}{\log \log k}}$ was conceived of by ChatGPT 5.5 Pro. The original paper that ChatGPT wrote is still available at~\cite{GPT}, for transparency's sake. Furthermore, in this same GitHub repository one can also find Lean formalizations of all theorems in this paper, including~\cite[Theorem 2]{Konyagin1999} that we apply. Apart from the external result on primes in short intervals from~\cite{BakerHarmanPintz2001}, these formalizations are fully self-contained. They were obtained by the automated theorem proving tool Aristotle from Harmonic~\cite{ari}.

\end{document}